\documentclass{article}
\usepackage{amsmath,amssymb,amsthm,stmaryrd,makeidx}
\usepackage[all]{xy}

\DeclareMathOperator{\id}{id}

\DeclareMathOperator{\mor}{Mor}

\DeclareMathOperator{\spec}{Spec}

\DeclareMathOperator{\ob}{ob}

\DeclareMathOperator{\pts}{pts}

\newcommand{\cat}[1]{\mathbf{#1}}

\newcommand{\sets}{\mathbf{Sets}}

\newcommand{\grmcat}[1]

\input xypic

\newtheorem{lemma}{Lemma}

\newtheorem{definition}{Definition}

\newtheorem{example}{Example}

\makeindex

\begin{document}
\author{Arvid Siqveland}
\title{Schemes of Objects in Abelian Categories}

\maketitle

\begin{abstract} In the article \cite{S250} we gave a natural definition of ordinary schemes based on the fact that the localization of a ring in a maximal ideal is a local representation of the corresponding function field. In this text, we replace the category of rings with a general locally small category $\cat C$, we consider a subcategory $\cat B\subset C$ of base-points, and assume that each $X\in\ob\cat C$ that contains $P\in\ob\cat B,$ i.e. there is a morphism $P\rightarrow X,$ there exists a local representing object $X_P.$ Assuming that coproducts exists, we can use the construction of ordinary schemes to construct schemes of objects in any such category.

\end{abstract}

\section{Introduction}

Let $A$ be a commutative ring. Let $\tilde O_X$ be the presheaf on $X=\spec A$ defined by $\tilde O_X(U)=\prod_{\mathfrak p\in U}A_{\mathfrak p},$ and consider the ring homomorphism $\rho=\prod\rho_{\mathfrak p}:A\rightarrow \tilde O_X(U).$ Define the presheaf $O_X$ by letting $O_X(U)$ be the subring of $\tilde O_X(U)$ generated by the subset $\rho(A)$ together with the elements $\rho(a)^{-1}$ whenever $\rho(a)$ is invertible. In the article \cite{S250} we proved that $O_X$ is a sheaf isomorphic to the sheaf of regular functions on $X,$ that is, $\mathcal O_X=O_X.$

In this text, we prove that we can replace the localization of rings, in prime ideals,  with a generalized localization of objects (in some locally small categories), in base-points. When the base-points is a subset of the category, and when there is introduced a topology on that set, we can define a Scheme structure of objects in the category. When the categorical product of the localized objects differ from the cartesian product of the corresponding sets, we choose to call this associative schemes of objects because the points are associated as simple modules over associative rings, see \cite{S244}.

\section{Base Points and Localization}

Let $\cat C$ be a locally small category, i.e., the collections $\mor(C_1,C_2)$ for $C_1,C_2\in\ob\cat C,$ are sets. Fix an object $P\in\ob\cat C.$

\begin{definition}{(The category of $P$-points.)} For an object $X\in\ob\cat C$ we define the category $\cat{Pts}_P(X),$ the category of $P$-points in $X,$ as the category with objects the set $$\ob\cat{Pts}_P(X)=\operatorname{pts}_P(X)=\mor(P,X),$$ and where a morphism from $x:P\rightarrow X$ to $y:P\rightarrow X$ is a commutative diagram $$\xymatrix{X\ar[rr]^\phi&&X\\&P\ar[ul]^x\ar[ur]_y&}$$
\end{definition}

Let $B\subseteq\ob\cat C$ be a sub-collection of the objects in $\cat C$ which we will call base-points. We extend the definition above.

\begin{definition}{(The category of $B$-points.)} For an object $X\in\ob\cat C$ we define the category $\cat{Pts}_B(X),$ the category of base-points in $X,$ as the category with objects the set $$\ob\cat{Pts}_B(X)=\operatorname{pts}_B(X)=\coprod_{P\in B}\mor(P,X),$$ and where a morphism from $x:P_1\rightarrow X$ to $y:P_2\rightarrow X$ is a commutative diagram $$\xymatrix{X\ar[r]^\phi&X\\P_1\ar[u]^x\ar[r]&P_2\ar[u]_y}$$
\end{definition}

We notice that $\cat{Pts}_B(X)$ is a concrete category. Consider the category $\cat C$ with fixed base points $B\subseteq\ob\cat C.$ For each object $C$ we want to fix a unique base-point, if it exists, $c\in\pts_B(C).$ If we define $E(B)$ as the collections of $D\in\ob\cat C$ such that $\pts_B(D)=\emptyset,$ this is the function $f:\ob\cat C\setminus E(B)\rightarrow\underset{C\in\ob\cat C\setminus E(B)}\coprod\pts_B(C)$ giving each object $C$  a fixed base-point.

\begin{definition}{(Subcategories of base-pointed objects.)}
Let $\overline{\cat C}_B$ denote the category with objects $\{(C,c)|C\in\ob\cat C,\pts_B(C)\neq\emptyset, c=f(C)\in\operatorname{pts}_B(C)\}$ and with morphisms $\phi:(C,c)\rightarrow (D,d)$ the commutative diagrams \begin{equation}\label{diag1}
\xymatrix{C\ar[r]^\phi&D\\P_1\ar[u]^c\ar[r]&P_2\ar[u]_d}\end{equation} We will call any full subcategory $\cat C_B\subseteq\overline{\cat C}_B$ a base-pointed subcategory of $\cat C.$
\end{definition}

In a commutative diagram (\ref{diag1}) we will use the notation $\phi(c)=d.$
Given an object $X\in\ob\cat C$ and a base-point  $x_P:P\rightarrow X, P\in B.$ Let $\cat C_B$ be a base-pointed subcategory of $\cat C,$ and consider the contravariant functor $F:\cat C_B\rightarrow\sets$ defined by $F(Y,y_Q)=\{\rho:Y\rightarrow X|\rho(y_Q)=x_P\}.$ This functor is represented by $((X_x,x_R),\rho)$ if the morphism $\rho$ fitting in the below diagram  is unique in the sense given in Definition \ref{repdef}.
$$\xymatrix{X_x\ar[r]^\rho&X\\R\ar[u]^{x_R}\ar[r]&P\ar[u]_{x_P}}$$

\begin{definition}\label{repdef} Let $X\in\ob\cat C$ and $x_P:P\rightarrow X.$ Then the localization of $X$ in $x,$ if it exists, is defined as a base-pointed object $(X_x,x_R)\in\ob\cat C_P$ characterized by the following universal property: There is a morphism $\rho:X_x\rightarrow X$ such that $\rho(x_R)=x_P,$ and if  $(L,l_Q)$ is a base-pointed object with a morphism $\gamma:L\rightarrow X$ such that $\gamma(l_Q)=x_P,$ there exists a unique morphism $\kappa:L\rightarrow X_x$ such that $\kappa(l_Q)=x_R.$
\end{definition}

We notice that the existence of localization in a $P$-point is dependent on the choice of a base-pointed subcategory, sometimes called localizing subcategories by several authors.

\begin{example}\label{Groupexample} In the category $\cat{Grps}$ of groups we fix the free (abelian) group on one element $P=\langle x\rangle.$ For every group $G$ we have a bijection $G\simeq\mor(P,G)$ as every group-homomorphism $\phi:P\rightarrow G$ is determined by its value on $x.$ In this case we consider the  $P$-pointed subcategory with objects $$\ob\cat C_P=\{(\langle x\rangle/\langle x^p\rangle,x)|p\text{ prime}\}.$$
Thus in this example, we let the base-points be the single object $P,$ that is $X=\{P\}.$ A morphism in this subcategory is a homomorphism $f:\langle x\rangle/\langle x^q\rangle\rightarrow\langle x\rangle/\langle x^p\rangle$ implying that $p=q.$ Thus For any group $G$ the localization $G_x$ in $x$ is the group $\langle x\rangle/\langle x^p\rangle$ where $p$ is the smallest prime such that $x^p=1.$ Notice the close connection to $p$-groups.
\end{example}

\begin{example} 
Let $k$ be a field and let $\cat{CAlg}_k$ denote the category of commutative $k$-algebras. As usual, we choose to consider the opposite category $\cat{CAlg}_k^o,$ 
and fix the object $P=k.$ Again, we let the base-points be the set $\{P\}$ with one element. As $P$-pointed subcategory we choose the category of all local $k$-algebras $A_\mathfrak m$ such that $A/\mathfrak m\simeq k.$
Then the localization of $A$ in $x:A\rightarrow k, \mathfrak m=\ker x,$ is a local algebra $A_{\mathfrak m}$ with maximal ideal $\mathfrak n,$ characterized by the following universal property: There exists a homomorphism $\rho:A\rightarrow A_{\mathfrak m}$ such that $\rho^{-1}(\mathfrak n)=\mathfrak m,$ and if there is another local algebra $B$ with maximal ideal $\mathfrak q$ and a homomorphism $\gamma:A\rightarrow B$ such that $\gamma^{-1}(\mathfrak q)=\mathfrak m$ then there is a unique morphism $\phi:A_{\mathfrak m}\rightarrow B$ with $\phi^{-1}(\mathfrak q)=\mathfrak n.$
\end{example}

\begin{example}
Let $k$ be a field and let $\cat{Vec}_k$ be the category of vector spaces over $k.$ As base-points in $\cat{Vec}_k$ we choose the single object $P=k,$ and as base-pointed subcategory, we choose the single object $(k,\id).$ Then the localization of a vector space $V$ in a $P$-point $x_{P}$ is isomorphic to the one-dimensional vector space $\operatorname{Span}(v_p), v_p=x_{P}(1).$ 
\end{example}

\begin{example}
The easiest example is maybe in the category of sets. Then we choose a set with one element $P=\{0\}$ as the only base-point, and the base-pointed subcategory as the sets with exactly one element. Then a localization of a set $S$ in $x_P$ is the one pointed set $\{x_P(0)\}.$
\end{example}

\begin{example}
Let $\cat{Daff}$ be the category of differentiable manifolds diffeomorphic to $\mathbb R^n$ for some $n\in\mathbb N.$ Morphisms in this category are $C^\infty$-functions. As base-points $B$ we choose the manifold $\mathbb R$ and as base-pointed subcategory the one-dimensional manifolds, i.e., diffeomorphic to $\mathbb R.$ The localization of an affine manifold $M$ in a point $p\in M$ is diffeomorphic to $\mathbb R.$
\end{example}

The next example (and Example \ref{Groupexample}) proves the richer structure in algebra.

\begin{example} Let $\cat{CRing}$ be the category of commutative rings with unit. The base-points are formed by the collection of fields $k$, and we let the base-pointed subcategory be the subcategory of local rings $L$ with base point $\kappa:L\rightarrow L/\mathfrak m$ where $\kappa$ is the quotient map and $\mathfrak m$ is the unique maximal ideal in $L.$ For a ring $A$ with base-point $x:A\rightarrow k$ we have $\mathfrak p=\ker x$ and and so there is a homomorphism $\rho:A\rightarrow A_{\mathfrak p}$ with $\rho^{-1}(\mathfrak p A_{\mathfrak p})=\mathfrak p,$ and such that if $B$ is any other local ring with this property, there is a unique morphism $A_{\mathfrak p}\rightarrow B.$ This is the ordinary definition  localization of the ring $A$ in the prime ideal $\mathfrak p.$  
\end{example}

\section{Schemes of Objects}

Let $\cat C$ be a locally small category with a fixed collection of base-points $B\subseteq\ob\cat C$ and a fixed base-pointed subcategory $\cat C_B.$ Assume that for each object $X$ in $\cat C$ the localization $X_x$ in all base-points $x$ exists, and that the localizations belong to a category where direct products and coproducts exist. Finally, we will assume that images and coimages exists in $\cat C.$ This is the case for all abelian categories, in particular in the examples above.

\begin{definition}{Image and coimage.}
Given a morphism $f:X\rightarrow Y.$ Then the image of $f,$ if it exists, is a monomorphism $m:I\rightarrow Y$ satisfying the following universal property:
\begin{itemize}
\item[1)] There exists a morphism $e:X\rightarrow Y$ such that $f=m\circ e.$
\item[2)] For any object $I'$ with a morphism $e':X\rightarrow I'$ and a monomorphism $m':I'\rightarrow Y$ such that $f=m'\circ e',$ there exists a unique morphism $v:I\rightarrow I'$ such that $m=m'\circ v.$
\end{itemize}
The coimage of $f,$ if it exists, is an epimorphism $c:X\rightarrow C$ satisfying the following universal property:
\begin{itemize}
\item[1)] There exists a morphism $m:C\rightarrow Y$ such that $f=m\circ e.$
\item[2)] For any object $C'$ with a epimorphism $e':X\rightarrow C'$ and a morphism $m':C'\rightarrow Y$ such that $f=m'\circ e',$ there exists a unique morphism $v:I'\rightarrow I$ such that $m=v\circ e'=e.$
\end{itemize}
\end{definition}

In the following, we treat the covariant and the contravariant categories as two separate cases, starting with the covariant case.

For any object $X$ in $\cat C,$ let $O(X)=\underset{x\in B}\coprod X_x.$ By the universal property of coproducts, there is a unique morphism $$\gamma:O(X)\rightarrow X,$$ and we put $\mathcal O(X)=C(\gamma)$ where $C(\gamma)$ is the coimage of $\gamma.$ 

In the opposite category, put $O(X)=\underset{x\in B}\prod X_x.$ Then there is a unique morphism $$\gamma:X\rightarrow O(X),$$ and we put $\mathcal O(X)=I(\gamma)$ where $I(\gamma)$ is the image of $\gamma.$

Let $X$ be an object in $\cat C.$ When there is given a morphism $U\rightarrow X$ in $\cat C,$ we have a map of sets $\pts_B(U)\rightarrow\pts_B(X).$ If this last map is injective, we call $U$ a sub object of $X.$ If there is given a topology on $\pts(X)$ and $\pts(U)$ is open in $\pts X,$ we will say that $U$ is an open sub object of $X.$ 

\begin{definition} With the data above, we call $\mathcal O(X)$ the global object of $X$ over $B.$ If $\mathcal O(X)\simeq X,$ we call $X$ an affine object in $\cat C.$ If $\pts_B(X)$ can be given a topology such that $X$ has a covering of open sub-objects $U\subseteq X$ such that $\mathcal O(U)\simeq U,$ we call $X$ a scheme of objects in $\mathcal C.$
\end{definition}

We remark that this definition works for any covering of $X,$ not necessarily a topology.

\section{Associative Schemes of objects}

We consider the locally small category $\cat C$ with a fixed class of base-points, i.e. a fixed base-pointed subcategory $\cat C_B$ of $\cat C.$ 

From our earlier work, e.g. \cite{ELS17}, it is clear that in general, if $M=\{x_I\}_{i=1}^r$ is a class of $r>1$ different base points, the existence of localizations $X_x$ for each $x\in M$ does not ensure the existence of the categorical product $\underset{x\in M}\prod X_x.$

By $E^r(B)$ denote those objects in $\cat C$ which do not contain $r$ different $B$-points.

\begin{definition}{The category of $r$-pointed objects.} Let $\overline{\cat C}_B^{\leq\infty}$ denote the category with objects $\{(C,M)|C\notin E^r(B),M=\{x_i\}_{i=1}^r\subseteq B\},$ and with morphisms $$(C,\{x_i\}_{i=1}^r)\rightarrow (D,\{y_i\}_{i=1}^s),\ s\leq r,$$ the commuting diagrams $$\xymatrix{C\ar[r]&D\\\{x_i\}_{i=1}^s\ar[r]_{\iota}\ar[u]&\{y_i\}_{i=1}^s\ar[u]}$$ with $\iota$ a bijection. We call any full subcategory $\cat C^{\leq\infty}_B$ of $\overline{\cat C}^{\leq\infty}_B$ a multi-pointed subcategory of $\cat C.$
\end{definition}

\begin{definition}
Let $X\in\ob\cat C$ and $M=\{x_i\}_{i=1}^r\subseteq \pts_B(X).$ Then the localization of $X$ in $M$ if it exists, is defined as an $r$-pointed object $(X_M,\{m_i\}_{i=1}^r)\in\ob\cat C_B$ characterized by the following universal property: There is a morphism $\rho:X_M\rightarrow X$ such that $\rho(m_i)=x_i,\ 1\leq i\leq r,$ and if  $(L,\{l_i\}_{i=1}^r)$ is an $r$-pointed object with morphisms $\gamma:L\rightarrow X$ such that $\gamma(l_i)=x_i,\ 1\leq i\leq r,$ there exists a unique morphism $\kappa:L\rightarrow X_M$ such that $\kappa(l_i)=m_i,\ 1\leq i\leq r.$ 
\end{definition}

Again we treat the covariant and the contravariant categories as two separate cases, starting with the covariant case.

\begin{lemma} Let $X\in\ob C$ and order the finite sets $M=\{x_i\}_{i=1}^r\subseteq B$ by inclusion. Then $$\underset{\underset{M\subseteq\pts_B(X)}\longrightarrow}\lim X_M=\underset{M\subseteq\pts_B(X)}\coprod X_M$$
\end{lemma}

\begin{proof} By definition of the coproduct, for every $M,$ there is a morphism $i_M:X_M\rightarrow \underset{M\subseteq\pts_B(X)}\coprod X_M.$ Also, by definition of the inductive limit, for every $M$ there is a morphism $j_M:X_M\rightarrow\underset{\underset{M\subseteq\pts_B(X)}\longrightarrow}\lim X_M.$ Taking eventual unions of finite subsets, there exists a unique isomorphism $\underset{M\subseteq\pts_B(X)}\coprod X_M\rightarrow\underset{\underset{M\subseteq\pts_B(X)}\longrightarrow}\lim X_M$
\end{proof}

We remark that the lemma says that the coproduct of the $X_M$ is the injective limit over finite sets $M$ partially ordered by inclusion, and that the product is the corresponding projective limit.

For any object $X$ in $\cat C,$ let $O(X)=\underset{\underset{M\text{ finite}}{M\subseteq\pts_B(X)}}\coprod X_M.$ By the universal property of coproducts, there is a unique morphism $$\gamma:O(X)\rightarrow X,$$ and we put $\mathcal O(X)=C(\gamma)$ where $C(\gamma)$ is the coimage of $\gamma.$ 

In the opposite category, put $O(X)=\underset{\underset{M\text{ finite}}{M\subseteq \pts_B(X)}}\prod X_M.$ Then there is a unique morphism $$\gamma:X\rightarrow O(X),$$ and we put $\mathcal O(X)=I(\gamma)$ where $I(\gamma)$ is the image of $\gamma.$

Let $X$ be an object in $\cat C.$ When given a morphism $U\rightarrow X$ in $\cat C$ we have a map of sets $\pts_B(U)\rightarrow\pts_B(X).$ If this last map is injective, we call $U$ a sub object of $X.$ If there is given a topology on $\pts(X)$ and $\pts(U)$ is open in $\pts X,$ we will say that $U$ is an open sub object of $X.$ 

\begin{definition} With the data above, we call $\mathcal O(X)$ the global object of $X$ over $B.$ If $\mathcal O(X)\simeq X,$ we call $X$ an affine object in $\cat C.$ If $\pts_B(X)$ can be given a topology such that $X$ has a covering of open sub-objects $U\subseteq X$ such that $\mathcal O(U)\simeq U,$ we call $X$ a scheme of objects in $\mathcal C.$
\end{definition}

\end{document}